\documentclass[10pt]{article}
\usepackage[latin1]{inputenc}
\usepackage{amsmath, amsthm, amssymb, setspace}
\begin{document}
\renewcommand{\Im}{\mathop{\rm Im }}
\renewcommand{\Re}{\mathop{\rm Re }}
\newcommand{\ra}{\mathop{\rightarrow }}
\newcommand{\supp}{\mathop{\rm supp}}
\newcommand{\sgn}{\mathop{\rm sgn }}
\newcommand{\card}{\mathop{\rm card }}
\newcommand{\KM}{\mbox{\rm KM}}
\newcommand{\diam}{\mathop{\rm diam}}
\newcommand{\diag}{\mathop{\rm diag}}
\newcommand{\tr}{\mathop{\rm tr}}
\newcommand{\Tr}{\mathop{\rm Tr}}
\newcommand{\dd}{\mathop{\rm d}}
\newcommand{\id}{\mbox{\rm1\hspace{-.2ex}\rule{.1ex}{1.44ex}}
   \hspace{-.82ex}\rule[-.01ex]{1.07ex}{.1ex}\hspace{.2ex}}
\renewcommand{\P}{\mathop{\rm Prob}}
\newcommand{\V}{\mathop{\rm Var}}
\newcommand{\cps}{{\stackrel{{\rm p.s.}}{\longrightarrow}}}
\newcommand{\limm}{\mathop{\rm l.i.m.}}
\newcommand{\cloi}{{\stackrel{{\rm loi}}{\rightarrow}}}
\newcommand{\bra}{\langle\,}
\newcommand{\ket}{\,\rangle}
\newcommand{\obl}{/\!/}
\newcommand{\mapdown}[1]{\vbox{\vskip 4.25pt\hbox{\bigg\downarrow
  \rlap{$\vcenter{\hbox{$#1$}}$}}\vskip 1pt}}
\newcommand{\tab}{&\!\!\!}

\newcommand{\tabb}{&\!\!\!\!\!}
\renewcommand{\d}{\displaystyle}
\newcommand{\epreuve}{\hspace{\fill}\qed}
\newcommand{\demo}{{\par\noindent{\em D\'emonstration~:~}}}
\newcommand{\solu}{{\par\noindent{\em Solution~:~}}}
\newcommand{\NB}{{\par\noindent{\bf Remarque~:~}}}
\newcommand{\const}{{\rm const}}
\newcommand{\cA}{{\cal A}}
\newcommand{\cB}{{\cal B}}
\newcommand{\cC}{{\cal C}}
\newcommand{\cD}{{\cal D}}
\newcommand{\cE}{{\cal E}}
\newcommand{\cF}{{\cal F}}
\newcommand{\cG}{{\cal G}}
\newcommand{\cH}{{\cal H}}
\newcommand{\cI}{{\cal I}}
\newcommand{\cJ}{{\cal J}}
\newcommand{\cK}{{\cal K}}
\newcommand{\cL}{{\cal L}}
\newcommand{\cM}{{\cal M}}
\newcommand{\cN}{{\cal N}}
\newcommand{\cO}{{\cal O}}
\newcommand{\cP}{{\cal P}}
\newcommand{\cQ}{{\cal Q}}
\newcommand{\cR}{{\cal R}}
\newcommand{\cS}{{\cal S}}
\newcommand{\cT}{{\cal T}}
\newcommand{\cU}{{\cal U}}
\newcommand{\cV}{{\cal V}}
\newcommand{\cW}{{\cal W}}
\newcommand{\cX}{{\cal X}}
\newcommand{\cY}{{\cal Y}}
\newcommand{\cZ}{{\cal Z}}
\newcommand{\bA}{{\bf A}}
\newcommand{\bB}{{\bf B}}
\newcommand{\bC}{{\bf C}}
\newcommand{\bD}{{\bf D}}
\newcommand{\bE}{{\bf E}}
\newcommand{\bF}{{\bf F}}
\newcommand{\bG}{{\bf G}}
\newcommand{\bH}{{\bf H}}
\newcommand{\bI}{{\bf I}}
\newcommand{\bJ}{{\bf J}}
\newcommand{\bK}{{\bf K}}
\newcommand{\bL}{{\bf L}}
\newcommand{\bM}{{\bf M}}
\newcommand{\bN}{{\bf N}}
\newcommand{\bP}{{\bf P}}
\newcommand{\bQ}{{\bf Q}}
\newcommand{\bR}{{\bf R}}
\newcommand{\bS}{{\bf S}}
\newcommand{\bT}{{\bf T}}
\newcommand{\bU}{{\bf U}}
\newcommand{\bV}{{\bf V}}
\newcommand{\bW}{{\bf W}}
\newcommand{\bX}{{\bf X}}
\newcommand{\bY}{{\bf Y}}
\newcommand{\bZ}{{\bf Z}}
\newcommand{\bu}{{\bf u}}
\newcommand{\bv}{{\bf v}}
\newfont{\msbm}{msbm10 scaled\magstep1}
\newfont{\msbms}{msbm7 scaled\magstep1} 
\newcommand{\bbA}{\mbox{$\mbox{\msbm A}$}}
\newcommand{\bbB}{\mbox{$\mbox{\msbm B}$}}
\newcommand{\bbC}{\mbox{$\mbox{\msbm C}$}}
\newcommand{\bbD}{\mbox{$\mbox{\msbm D}$}}
\newcommand{\bbE}{\mathbb{E}}
\newcommand{\bbF}{\mbox{$\mbox{\msbm F}$}}
\newcommand{\bbG}{\mbox{$\mbox{\msbm G}$}}
\newcommand{\bbH}{\mbox{$\mbox{\msbm H}$}}
\newcommand{\bbI}{\mbox{$\mbox{\msbm I}$}}
\newcommand{\bbJ}{\mbox{$\mbox{\msbm J}$}}
\newcommand{\bbK}{\mbox{$\mbox{\msbm K}$}}
\newcommand{\bbL}{\mbox{$\mbox{\msbm L}$}}
\newcommand{\bbM}{\mbox{$\mbox{\msbm M}$}}
\newcommand{\bbN}{\mathbb{N}}
\newcommand{\bbO}{\mbox{$\mbox{\msbm O}$}}
\newcommand{\bbP}{\mathbb{P}}
\newcommand{\bbQ}{\mbox{$\mbox{\msbm Q}$}}
\newcommand{\bbR}{\mathbb{R}}
\newcommand{\bbS}{\mbox{$\mbox{\msbm S}$}}
\newcommand{\bbT}{\mbox{$\mbox{\msbm T}$}}
\newcommand{\bbU}{\mbox{$\mbox{\msbm U}$}}
\newcommand{\bbV}{\mbox{$\mbox{\msbm V}$}}
\newcommand{\bbW}{\mbox{$\mbox{\msbm W}$}}
\newcommand{\bbX}{\mbox{$\mbox{\msbm X}$}}
\newcommand{\bbY}{\mbox{$\mbox{\msbm Y}$}}
\newcommand{\bbZ}{mathbb{Z}}
\newcommand{\bbsA}{\mbox{$\mbox{\msbms A}$}}
\newcommand{\bbsB}{\mbox{$\mbox{\msbms B}$}}
\newcommand{\bbsC}{\mbox{$\mbox{\msbms C}$}}
\newcommand{\bbsD}{\mbox{$\mbox{\msbms D}$}}
\newcommand{\bbsE}{\mbox{$\mbox{\msbms E}$}}
\newcommand{\bbsF}{\mbox{$\mbox{\msbms F}$}}
\newcommand{\bbsG}{\mbox{$\mbox{\msbms G}$}}
\newcommand{\bbsH}{\mbox{$\mbox{\msbms H}$}}
\newcommand{\bbsI}{\mbox{$\mbox{\msbms I}$}}
\newcommand{\bbsJ}{\mbox{$\mbox{\msbms J}$}}
\newcommand{\bbsK}{\mbox{$\mbox{\msbms K}$}}
\newcommand{\bbsL}{\mbox{$\mbox{\msbms L}$}}
\newcommand{\bbsM}{\mbox{$\mbox{\msbms M}$}}
\newcommand{\bbsN}{\mbox{$\mbox{\msbms N}$}}
\newcommand{\bbsO}{\mbox{$\mbox{\msbms O}$}}
\newcommand{\bbsP}{\mbox{$\mbox{\msbms P}$}}
\newcommand{\bbsQ}{\mbox{$\mbox{\msbms Q}$}}
\newcommand{\bbsR}{\mbox{$\mbox{\msbms R}$}}
\newcommand{\bbsS}{\mbox{$\mbox{\msbms S}$}}
\newcommand{\bbsT}{\mbox{$\mbox{\msbms T}$}}
\newcommand{\bbsU}{\mbox{$\mbox{\msbms U}$}}
\newcommand{\bbsV}{\mbox{$\mbox{\msbms V}$}}
\newcommand{\bbsW}{\mbox{$\mbox{\msbms W}$}}
\newcommand{\bbsX}{\mbox{$\mbox{\msbms X}$}}
\newcommand{\bbsY}{\mbox{$\mbox{\msbms Y}$}}
\newcommand{\bbsZ}{\mbox{$\mbox{\msbms Z}$}}

%
\def\eurtoday{\number\day \space\ifcase\month\or
 January\or February\or March\or April\or May\or June\or
 July\or August\or September\or October\or November\or December\fi
 \space\number\year}
%
%
\def\aujourdhui{\number\day \space\ifcase\month\or
 janvier\or f{\'e}vrier\or mars\or avril\or mai\or juin\or
 juillet\or ao{\^u}t\or septembre\or octobre\or novembre\or d{\'e}cembre\fi
 \space\number\year}
\textheight=21cm
\textwidth=16cm 
\voffset=-1cm
\newtheorem{theo}{Theorem}[section]
\newtheorem{pr}{Proposition}[section]
\newtheorem{cor}{Corollary}[section]
\newtheorem{lem}{Lemma}[section]
\newtheorem{defn}{Definition}[section]
\hoffset=-1,5cm
\parskip=4mm
\title{On the local time of random walks associated with Gegenbauer polynomials}
\author{Nadine Guillotin-Plantard \thanks{Universit\'e de Lyon ; Universit\'e Lyon 1 ; INSA de Lyon, F-69621 ; Ecole Centrale de Lyon ; CNRS, UMR5208,
Institut Camille Jordan, 43 bld du 11 novembre 1918, F-69622 Villeurbanne-Cedex, France, e-mail: nadine.guillotin@univ-lyon1.fr}}
\date{}
\maketitle
~\\
{\it Key words:} random walk, local time, local limit theorem, Gegenbauer polynomials, Markov chain, transition kernel, recurrence, transience, birth and death process. \\
~\\
{\it AMS Subject Classification:} 60J10, 60J55, 60F05, 60B99.

\begin{abstract}
The local time of random walks associated with Gegenbauer polynomials $P_n^{(\alpha)}(x),\ x\in [-1,1]$ is studied in the recurrent case: $\alpha\in\ [-\frac{1}{2},0]$.   
When $\alpha$ is nonzero, the limit distribution is given in terms of a Mittag-Leffler distribution. 
The proof is based on a local limit theorem for the random walk associated with Gegenbauer polynomials. As a by-product, we derive the limit distribution of the local time of some particular birth and death Markov chains on $\bbN$.
\end{abstract}

\section{Introduction}
Random walks on hypergroups have been extensively studied over the last decades. A history of these  processes as well as the motivations for studying them are provided in \cite{Galhist}. We here restrict ourselves to discrete polynomial hypergroups~: Let $(\alpha_n)_{n\in\bbsN}$, $(\beta_n)_{n\in\bbsN}$ and $(\gamma_n)_{n\in\bbsN}$ be real sequences with the following properties: $\gamma_n > 0$,
$\beta_n \geq 0, \alpha_{n+1} > 0$ for all $n \in \bbN$, moreover $\alpha_0= 0$, and $\alpha_n + \beta_n + \gamma_n = 1$ for all $n \in \bbN$. We
define the sequence of polynomials $(P_n)_{n\in\bbsN}$ by $P_0(x) = 1$, $P_1(x) = x$, and by the recursive
formula
$$xP_n(x) = \alpha_n P_{n-1}(x)+\beta_n P_n(x)+\gamma_n P_{n+1}(x)$$
for all $n \geq 1$ and $x \in \bbR$. In this case there exist constants $c(n,m,k), n, m, k \in \bbN$ such that the following {\it linearization formula}
$$P_nP_m =\sum_{k=|n-m|}^{n+m} c(n,m,k) P_k$$
holds for all $n, m \in \bbN$. Since $P_n(1) = 1$ for all $n\in \bbN$, we have for all $n, m \in \bbN$,
$$\sum_{k=|n-m|}^{n+m} c(n,m,k) = 1.$$
If the coefficients $c(n,m,k)$ are nonnegative for all $n, m, k
\in \bbN$, then a hypergroup structure on $\bbN$ is obtained from the generalized convolution $\star$ defined as follows: for all $n, m \in \bbN,$
$$\delta_n \star \delta_m =\sum_{k=|n-m|}^{n+m} c(n,m,k)\delta_k.$$
The resulting hypergroup $K = (\bbN,\star)$ is called {\it the discrete polynomial hypergroup associated with the sequence
$(P_n)_{n\in\bbsN}.$} Many classical families of orthogonal polynomials with respect to some positive measure on $[-1,1]$ satisfy a linearization formula with nonnegative coefficients. 
A random walk with distribution $\mu\in {\cal M}_1(\bbN)$ on the hypergroup $(\bbN,\star)$ is then defined as a homogeneous Markov chain on $\bbN$ with Markov kernel given by
$$p(x,y)=\delta_x\star\mu(y),\ \ x,y\in\bbN.$$
This Markov chain is called {\it random walk associated with the sequence of polynomials  $(P_n)_{n\in\bbsN}$}.\\* 
Limit theorems (law of large numbers, central limit theorem, local limit theorems, large deviation principle, iterated logarithm law,...) for these processes were earlier investigated by M. Ehring \cite{Her, Her2}, M. Voit \cite{Voi1, Voi2, Voi3, Voi4}.  
L. Gallardo {\it et al} (\cite{Gal1, Gal}), Y. Guivarc'h {\it et al} \cite{gui}, M. Mabrouki \cite{Mab} have more specifically studied limit theorems for random walks associated with Gegenbauer polynomials.
In this paper we present some extensions of the theory developed by these authors by deriving a limit theorem for the local time of the random walks associated with Gegenbauer polynomials $(P_n^{(\alpha)})_{n\in\bbsN}$ for every $\alpha\in [-1/2,0]$.

The organization of the paper is as follows: We recall in Section 2 some generalities on Bessel processes and its local time. In Section 3, Gegenbauer polynomials as well as the definition of random walks associated with these polynomials are given. In Section 4, classical limit theorems for these processes are presented. Section 5 is devoted to the study of the local time of these Markov chains and some particular cases are considered.  

\section{Preliminaries on the Bessel process and its local time}
For every $\alpha\in\, [-1,+\infty[$, we will denote by $B^{(\alpha)}:=(B_{t}^{(\alpha)})_{t\in\bbsR_{+}}$ the unique solution of the stochastic differential equation
$$ X_t^2=X_0^2+2\int_0^t\sqrt{X_s^2}\ \mbox{\rm d}B_s + 2(\alpha +1)\, t$$
where  $(B_{t})_{t\in\bbsR_{+}}$ is the real Brownian motion. 
The parameter $\alpha$ is usually called the {\it index} of the Bessel process $B^{(\alpha)}$.  
The process $(B_{t}^{(\alpha)})_{t\in\bbsR_{+}}$ can also be defined as the $\bbR_{+}$-valued Feller diffusion whose infinitesimal generator ${\cal L}$ is defined as:
$${\cal L}f=\frac{1}{2}\frac{\rm{d}^2 f}{\rm{d}x^2} +\frac{2\alpha+1}{2 x} \frac{\rm{d}f}{\rm{d}x}$$
on the domain
$${\cal D}({\cal L})=\{f:\bbR_{+}\rightarrow\bbR;\ {\cal L}f\in C_b(\bbR_{+}),\ \lim_{x\downarrow 0} x^{2\alpha +1} f'(x)=0\}.$$
The Bessel process has the Brownian scaling property: for every $c>0$, the processes $(B_{ct}^{(\alpha)})_{t\in\bbsR_{+}}$ and $(\sqrt{c}B_t^{(\alpha)})_{t\in\bbsR_{+}}$ have the same law, when $B_0^{(\alpha)}\equiv 0$.

Let us fix $\alpha\in\ ]-1,0[$, it is well-known (see \cite{RY}) that there exists a jointly continuous family $(L_{t}^{(\alpha)}(x))_{x\in\bbsR_{+}, t \in\bbsR_{+}}$ of local times such that the occupation formula:
$$\int_0^t h( B_s^{(\alpha)})\ \mbox{\rm d}s =2\int_0^\infty  h(x)L_t^{(\alpha)}(x) x^{2\alpha+1}\ \mbox{\rm d}x$$
holds for every Borel function $h : \bbR_{+} \rightarrow \bbR_{+}.$
We could also take as a definition of $L_{t}^{(\alpha)}(0)$ the unique continuous increasing process such that :
$$(B_{t}^{(\alpha)})^{2 |\alpha|} - 2|\alpha|L_{t}^{(\alpha)}(0),\,  t  \geq 0, $$
 be a martingale. \\*
For every $\alpha>0$, the unique distribution with Laplace transform given by the Mittag-Leffler function (see \cite{FE} p. 453)
$$E_{\alpha}(x)= \sum_{p=0}^{\infty} \frac{(-x)^p}{\Gamma(p\alpha+1)},\ \ x\in\bbR_{+}$$
is called {\it Mittag-Leffler distribution} and denoted by ${\cal M}(\alpha)$. Here, $\Gamma(.)$ denotes the usual Gamma function. The $p-$th moment of ${\cal M}(\alpha)$ is equal to
$$\frac{p!}{\Gamma(\alpha p+1)}.$$
The probability density of ${\cal M}(\alpha)$ (see \cite{Man}) is equal to
$$f(x)= \frac{1}{\pi} \sum_{k=1}^{\infty} \frac{(-1)^{k-1}}{(k-1)!}\sin(\pi k\alpha)\Gamma(k\alpha) x^{k-1}, x\in \ ]0,+\infty[.$$
In \cite{Gra} (p. 567, Formula (R8)), the distribution of the random variable $L_1^{(\alpha)}(0)$ is given in terms of the Mittag-Leffler distribution, namely
\begin{equation}\label{loc}
\frac{2^{\alpha+1} \Gamma(\alpha+1)}{\Gamma(|\alpha|)}L_1^{(\alpha)}(0)\stackrel{\rm law}{=}{\cal M}(|\alpha|).
\end{equation}
\section{Random walk associated with Gegenbauer polynomials}
\subsection{Generalities on Gegenbauer Polynomials} 
{\it Gegenbauer polynomials} also called {\it ultraspherical polynomials} are defined on $[-1,1]$ for any $\alpha>-1$ by
$$P_n^{(\alpha)}(x) = \frac{(-1)^n}{2^n (\alpha+1)\ldots(\alpha+n)}(1-x^2)^{-\alpha} \frac{\mbox{\rm d}^{n}}{\mbox{\rm d}x^n} (1-x^2)^{n+\alpha}.$$ 
They satisfy the following orthogonality relations~:
\begin{equation}\label{ort}
\int_{-1}^1 P_n^{(\alpha)}(x)P_m^{(\alpha)}(x)\, \mbox{\rm d}\pi_{\alpha}(x)=\left\{\begin{array}{lll}
0 & if & n\neq m\\
(w_n^{(\alpha)})^{-1} & if & n=m
\end{array}
\right.
\end{equation}
where $ \mbox{\rm d}\pi_{\alpha}(x) = (1-x^2)^{\alpha} {\bf 1}_{[-1,1]}(x)\ \mbox{\rm d}x$
and \\
$-\ n\neq 0$:
\begin{equation}\label{ww}
w_n^{(\alpha)} = \frac{(2n+2\alpha+1) \Gamma(n+2\alpha+1)}{2^{2\alpha+1} \Gamma(n+1)
\Gamma(\alpha +1)^2}
\end{equation}
$-\  n=0$:
\begin{equation}\label{www}
w_{0}^{(\alpha)}=\left\{\begin{array}{ll}
\frac{(2\alpha+1) \Gamma(2\alpha+1)}{2^{2\alpha+1} \Gamma(\alpha+1)^2}& \mbox{\rm if} \ \alpha\neq -1/2\\
\ \ \ \ \ \ \ \ 1/\pi & \mbox{\rm otherwise} 
\end{array}
\right.
\end{equation}
These polynomials satisfy the following properties~:
\begin{eqnarray}\label{par}
P_0^{(\alpha)}(x)&\equiv& 1,\ \ P_1^{(\alpha)}(x)\equiv x\nonumber \\
P_n^{(\alpha)}(-x)&=&(-1)^n P_n^{(\alpha)}(x)\\
P_n^{(\alpha)} (1)&=& 1 \nonumber
\end{eqnarray}
and the multiplication formula
\begin{equation}
P_1^{(\alpha)}(x)P_n^{(\alpha)}(x) = \frac{n}{2n+2\alpha+1} P_{n-1}^{(\alpha)}(x) 
+\frac{n+2\alpha+1}{2n+2\alpha+1} P_{n+1}^{(\alpha)}(x)
\end{equation}
for every $\alpha>-1$, every $n\in\bbN^{\star}$ and every $x\in\ [-1,1]$.\\*
More generally, when $\alpha\geq -\frac{1}{2},$ we have, for $m\leq n$,
\begin{equation}\label{CG}
P_m^{(\alpha)}(x)P_n^{(\alpha)}(x) = \sum_{r=0}^m C^{(\alpha)}(m,n,r) P_{n-m+2r}^{(\alpha)}(x)
\end{equation}
where the nonnegative coefficients $C^{(\alpha)}(m,n,r)$ are explicit (see \cite{Gal1}).\\*
Finally, when $\alpha>-\frac{1}{2}$, for any $n\geq 1$,
$$|P_n^{(\alpha)}(x)|<1,\quad  \forall x\in\ ]-1,+1[.$$

\subsection{Random walk associated with Gegenbauer polynomials}
Let ${\cal M}_1 (\bbN)$ be the set of probability measures $\mu = \sum_{n\in \bbsN}\mu(n)\delta_n $ on $\bbN$. 
Thanks to Formula $(\ref{CG})$, for every $\alpha\in [-1/2,+\infty[$,
we can define a generalized convolution denoted $\star$ as follows: 
$$ \delta_m\star\delta_n=\sum_{r}C^{(\alpha)}(m,n,r)\delta_{n-m+2r}$$
and more generally, if $\mu$, $\nu$ are in ${\cal M}_1(\bbN)$:
$$ \mu\star\nu = \sum_{m,n\in\bbsN}\mu(n)\nu(m)\delta_n\star\delta_m.$$
For each $x\in \bbN$ and for each subset $A$ of $\bbN$, we can define the transition kernel from $\bbN$ to $\bbN$~:
$$P(x,A)=(\delta_{x}\star\mu) (A).$$
{\it The random walk associated with Gegenbauer polynomials} is defined as the Markov chain with state space $\bbN$ and transition kernel given by $P$ and will be denoted by $(S_n)_{n\geq 0}$. In the sequel, for the sake of clarity, we will omit in the notation of the Markov chain the index $\alpha$ and the measure $\mu$ from which the process is defined.\\* 
The probability to be in a subset $A$ of $\bbN$ at time $n$ when departing from state $x$ at time $0$ is then given by 
$$ P^{(n)}(x,A) = (\delta_x\star{\mu}^{(n)})(A)$$
with the notation $ {\mu}^{(n)} = \mu\star\ldots\star\mu$ ($n$ times).\\*
It is worth noticing that if the distribution $\mu$ is the Dirac mass at point 1, then the random walk associated with Gegenbauer polynomials with index $\alpha$ is the birth and death Markov chain on $\bbN$ with transition probabilities given by 
$$p(0,1)=1 $$
and
$$p(i,i+1)=1-p(i,i-1)=\frac{1}{2}\Big(1+\frac{\lambda}{i+\lambda}\Big)$$
where $\lambda=\alpha+\frac{1}{2}\in [0,+\infty[$.\\
A natural question is to know if a given Markov chain with state space $\bbN$ corresponds to a random walk associated with Gegenbauer polynomials. This is true if and only if the transition probabilities of the Markov chain satisfy the following relation (see \cite{Gal1})
\begin{eqnarray*}
\frac{i}{2(i+\lambda)}p(i-1,j)+\frac{i+2\lambda}{2(i+\lambda)}p(i+1,j)&=&\frac{j+2\lambda-1}{2(j+\lambda-1)}p(i,j-1)\\
&+&\frac{j+1}{2(j+\lambda+1)}p(i,j+1)
\end{eqnarray*}
for some $\lambda\in  [0,\frac{1}{2}]$.\\* 
The distribution $\mu$ is then given by $\mu(n)=p(0,n)$.\\*
Finally, we call {\it (generalized) Fourier transform} of $\mu\in {\cal M}_1(\bbN)$ the function $\hat{\mu}$ defined on $[0,\pi]$ by
$$\hat{\mu}(\theta)=\sum_{n\in\bbsN}\mu(n) P_{n}^{(\alpha)}(\cos(\theta)).$$
From orthogonality relations (\ref{ort}), the coefficient $\mu(n)$ of the measure $\mu$ can be obtained from $\hat{\mu}$ by the following formula
\begin{eqnarray}\label{four}
 \mu(n) &=& w_n^{(\alpha)} \int_0^{\pi} \hat{\mu}(\theta) P_n^{(\alpha)}(\cos(\theta))\, \sin^{2\alpha+1}(\theta)\, \mbox{\rm d}\theta.
\end{eqnarray}
In particular, $\hat{\delta}_n(\theta)=P_n^{(\alpha)}(\cos(\theta)) $ and thanks to Formula (\ref{CG}),
$$\widehat{(\delta_n\star \delta_m)}= \hat{\delta}_n\hat{\delta}_m.$$
More generally, for every $\mu,\nu\in{\cal M}_1(\bbN)$,
$$  \widehat{(\mu \star \nu)} =  \widehat{\mu}\  \widehat{\nu}.$$

\section{Limit theorems}
In this section, we denote by $(S_n)_{n\geq 0}$ the random walk associated with Gegenbauer polynomials (as defined in Section 3.2) with transition kernel given by $\delta_x\star \mu$ for some $\mu\in {\cal M}_1(\bbN)$. 
\subsection{A functional central limit theorem}
Let ${\cal D}={\cal D}([0,+\infty[)$ be the space of c\`adl\`ag functions on $\bbR_{+}$ endowed with the Skorohod topology. We denote by $(B_t^{(\alpha)})_{t\in\bbsR_{+}}$ the Bessel process on $\bbR_{+}$ of index $\alpha \in [-\frac{1}{2},+\infty[$ defined in Section 2. 
\begin{theo}\cite{Mab}
Let $\mu\in {\cal M}_1(\bbN)$ with a second order moment and
 $$C=\frac{1}{4(\alpha +1)}\sum_{n=1}^{\infty} \mu(n) n(n+2\alpha+1).$$
The sequence $ \Big(\frac{S_{[nt]}}{\sqrt{2Cn}}\Big)_{t\in\bbsR_{+}} $
converges in ${\cal D}$, as $n\rightarrow +\infty$, to the process $(B_t^{(\alpha)})_{t\in\bbsR_{+}}$.
\end{theo}

\subsection{A large deviation principle - A law of large numbers}
A large deviation principle for polynomial hypergroups was proved by Ehring \cite{Her}.
\begin{theo} \cite{Her}
Let $\mu\in {\cal M}_1(\bbN)$ with finite support.
Then, the sequence of random variables $(\frac{S_n}{n})_{n\geq 1}$ satisfies a large deviation principle of speed $n$ and good rate function 
$$I(x)=\left\{\begin{array}{lll}
               +\infty & if & x\notin [0,x_0] \\
\displaystyle\sup_{\lambda\geq 0}\{\lambda x - \log(\tilde{\mu}(\lambda))\} & if & x\in [0,x_0]
\end{array}
\right.$$ 
where $x_0=\max\{x\in\bbN\ |\ \mu(x)\neq 0\}$ and $\tilde{\mu}$ is the (generalized) Laplace transform of the measure $\mu$ (see \cite{Her} for the definition).
\end{theo}
In \cite{Her}, $x=0$ is proved to be the unique infimum point of the function $I$, then a weak law of large numbers for $(\frac{S_n}{n})_{n\geq 1}$ holds. 
\subsection{Local limit theorems}
By using a (generalized) Fourier calculus, a local limit theorem for the random walk associated with Gegenbauer polynomials for any $\alpha \geq -\frac{1}{2}$ was proved in \cite{Her2}.
\begin{theo}\label{llt}
Let us assume that $\mu$ is aperiodic (i.e. the support of $\mu$ is not a subset of $2\bbN$) with a finite second order moment.\\* 
Then, for every $x,y\in\bbN$, as $n\rightarrow+\infty$,
$$p^{(n)} (x,y) \sim  \frac{w_y^{(\alpha)}\ \Gamma(\alpha+1)}{2(Cn)^{\alpha+1}}$$
where 
$$C=\frac{1}{4(\alpha +1)}\sum_{n=1}^{\infty} \mu(n) n(n+2\alpha+1).$$
The random walk associated with Gegenbauer polynomials is then
$$\left\{\begin{array}{ll}
\mbox{\rm recurrent if} & \alpha\in\ [-\frac{1}{2},0].\\
\mbox{\rm transient if} & \alpha\in\ ]0,+\infty[.
\end{array}
\right.
$$
\end{theo}
The properties of recurrence/transience of the Markov chain $(S_n)_{n\in\bbsN}$ were established in \cite{gui} by computing the potential kernel of the Markov chain. 
The previous statement can be extended to the case when the measure $\mu$ is not aperiodic. We are only interested in the case $\mu=\delta_1$, but generalization to any periodic measure can easily be done. 
\begin{pr}\label{apero}
Assume that $\mu=\delta_1$, then for any $x,y\in\bbN$, as $n\rightarrow+\infty$,
$$p^{(n)}(x,y) \sim \left\{\begin{array}{ll}
w_y^{(\alpha)} 2^{\alpha+1} \Gamma(\alpha+1) n^{-(\alpha+1)} &\mbox{when}\ \ n+x+y \mbox{ is even.} \\
\ \ \ 0 & \mbox{otherwise.}
 \end{array} 
 \right.
 $$
\end{pr}
We now prove a local limit theorem with a new normalization in space following the same lines as the proof of Theorem \ref{llt} in \cite{Her2}. Therefore, in the proof, we just stress on points which differ.  
\begin{theo}\label{t2}
Let us assume that $\mu$ is aperiodic with a second order moment.\\* 
Then, for every $x\in\bbR_{+}^{\star}$, as $n\rightarrow +\infty$,
\begin{equation}\label{eq:1}
\sqrt{n}\ p^{(n)} (\lfloor x\sqrt{n}\rfloor,\lfloor x\sqrt{n} \rfloor) \sim \frac{x}{2C} e^{-\frac{x^2}{2C}} I_{\alpha}\left(\frac{x^2}{2C}\right)
\end{equation}
and
\begin{equation}\label{eq:2}
\sqrt{n}\ p^{(n)} (0,\lfloor x\sqrt{n}\rfloor) \sim \frac{x^{2\alpha+1}e^{-\frac{x^2}{4C}}}{2^{2\alpha+1} C^{\alpha+1}\Gamma(\alpha+1)}
\end{equation}
where $I_{\alpha}$ is the modified Bessel function of index $\alpha$.
\end{theo}
\begin{proof}
From Formula (\ref{four}),
\begin{eqnarray*}
\sqrt{n} p^{(n)} (\lfloor x\sqrt{n}\rfloor,\lfloor x\sqrt{n}\rfloor) &=& \sqrt{n}\ w_{\lfloor x\sqrt{n}\rfloor}^{(\alpha)} \int_0^{\pi} \hat{\mu}(\theta)^n \left(P_{\lfloor x\sqrt{n}\rfloor}^{(\alpha)}(\cos(\theta))  \sin^{\alpha+\frac{1}{2}}(\theta)\right)^2 \, {\rm d}\theta\\
&=&  w_{\lfloor x\sqrt{n}\rfloor}^{(\alpha)} \int_0^{\pi\sqrt{n}} \hat{\mu}\Big(\frac{\theta}{\sqrt{n}}\Big)^n \left(P_{\lfloor x\sqrt{n}\rfloor}^{(\alpha)}(\cos(\frac{\theta}{\sqrt{n}}))\sin^{\alpha+\frac{1}{2}} \Big(\frac{\theta}{\sqrt{n}}\Big)\right)^2 \, {\rm d}\theta
\end{eqnarray*}
using the change of variables $u=\sqrt{n}\theta$.\\*
The right-hand side is then decomposed as the sum of the following integrals~:
\begin{eqnarray*}
I_0(n)&=& w_{\lfloor x\sqrt{n}\rfloor}^{(\alpha)} \int_0^{\infty} \exp(-C\theta^2) \left(P_{\lfloor x\sqrt{n}\rfloor}^{(\alpha)}\Big(\cos\Big(\frac{\theta}{\sqrt{n}}\Big)\Big) \sin^{\alpha+\frac{1}{2}} \Big(\frac{\theta}{\sqrt{n}}\Big)\right)^2 \, \mbox{\rm d}\theta\\
I_1(n,A) &=& w_{\lfloor x\sqrt{n}\rfloor}^{(\alpha)} \int_0^{A} \Big[\hat{\mu}\Big(\frac{\theta}{\sqrt{n}}\Big)^n
-\exp(-C\theta^2)\Big]\left(P_{\lfloor x\sqrt{n}\rfloor}^{(\alpha)}\Big(\cos\Big(\frac{\theta}{\sqrt{n}}\Big)\Big) \sin^{\alpha+\frac{1}{2}} \Big(\frac{\theta}{\sqrt{n}}\Big)\right)^2 \, \mbox{\rm d}\theta\\
I_2(n,A) &=& - w_{\lfloor x\sqrt{n}\rfloor}^{(\alpha)} \int_A^{+\infty} \exp(-C\theta^2) \left(P_{\lfloor x\sqrt{n}\rfloor}^{(\alpha)}\Big(\cos\Big(\frac{\theta}{\sqrt{n}}\Big)\Big)\sin^{\alpha+\frac{1}{2}} \Big(\frac{\theta}{\sqrt{n}}\Big)\right)^2 \, \mbox{\rm d}\theta\\
I_3(n,A,r)&=&  w_{\lfloor x\sqrt{n}\rfloor}^{(\alpha)} \int_A^{r\sqrt{n}} \hat{\mu}\Big(\frac{\theta}{\sqrt{n}}\Big)^n \left(P_{\lfloor x\sqrt{n}\rfloor}^{(\alpha)}\Big(\cos\Big(\frac{\theta}{\sqrt{n}}\Big)\Big)\sin^{\alpha+\frac{1}{2}} \Big(\frac{\theta}{\sqrt{n}}\Big)\right)^2 \, \mbox{\rm d}\theta\\
I_4(n,r)&=&w_{\lfloor x\sqrt{n}\rfloor}^{(\alpha)} \int_{r\sqrt{n}}^{\pi\sqrt{n}} \hat{\mu}\Big(\frac{\theta}{\sqrt{n}}\Big)^n \left(P_{\lfloor x\sqrt{n}\rfloor}^{(\alpha)}\Big(\cos\Big(\frac{\theta}{\sqrt{n}}\Big)\Big) \sin^{\alpha+\frac{1}{2}} \Big(\frac{\theta}{\sqrt{n}}\Big)\right)^2 \, \mbox{\rm d}\theta
\end{eqnarray*}
We only give the way of estimating $I_0(n)$ for $n$ large. The integrals $I_j, j=1,\ldots,4$ can be proved to be negligible as in the proof of Theorem \ref{llt}.  
From the definition of the $w_x^{(\alpha)}$'s, we easily deduce that as $n\rightarrow+\infty$,  
$$w_{\lfloor x\sqrt{n}\rfloor}^{(\alpha)}\sim \frac{x^{2\alpha+1} n^{\alpha+\frac{1}{2}}}{2^{2\alpha} \Gamma(\alpha+1)^2}.$$
Moreover, from Formula 8.21.12 in \cite{SZ}, we get   
$$\lim_{n\rightarrow +\infty} P_{\lfloor x\sqrt{n}\rfloor}^{(\alpha)}\Big(\cos\Big(\frac{\theta}{\sqrt{n}}\Big)\Big)= \frac{2^{\alpha}\Gamma(\alpha+1)J_{\alpha}(\theta x)}{\theta^{\alpha} x^{\alpha}}$$
where $J_{\alpha}$ is the Bessel function of index $\alpha$ which yields (\ref{eq:1}) from dominated convergence theorem, by remarking that
$$\int_{\bbsR_{+}} e^{-C\theta^2} (\theta x) J_{\alpha}^2(\theta x) {\rm d}\theta=\frac{x}{2C} e^{-x^2/2C}I_{\alpha} (\frac{x^2}{2C}).$$
We obtain (\ref{eq:2}) by remarking that we have 
$$\int_{\bbsR_{+}} e^{-C\theta^2} (\theta x)^{\alpha+1} J_{\alpha}(\theta  x)\, {\rm d}\theta =\frac{x^{2\alpha+1}}{(2C)^{\alpha+1}} e^{-\frac{x^2}{4C}}.$$
\end{proof}

\section{Limit distribution of the local time}
The local time $(N_n(x))_{n\geq 0; x\in\bbsN}$ of the Markov chain $(S_n)_{n\in\bbsN}$ defined in Section 3.2 is equal to the number of times the chain visits the site $x$ up to time $n$, namely
$$N_n(x)=\sum_{k=0}^n {\bf 1}_{\{S_k=x\}}.$$
For every $x\in\bbN$, we denote by $\bbP_x$ the distribution of the Markov chain $(S_n)_{n\geq 0}$ starting from $x$ and by $\bbE_x$ the corresponding expectation. We prove in the case $\alpha\in [-\frac{1}{2},0]$ the following limit theorem for the local time $(N_n(x))_{n\geq 0; x\in\bbsN}$.
\begin{theo}\label{lt} 
Assume that $\mu$ is aperiodic with a finite second order moment.
\begin{itemize} 
\item When $\alpha\in[-\frac{1}{2},0[$, for every $x,y\in\bbN$, under $\bbP_x$,
\begin{equation}\label{res1}
\frac{N_n(y)}{n^{|\alpha|}}\stackrel{\mathcal{L}}{\longrightarrow} \frac{w_y^{(\alpha)}\Gamma(\alpha+1) \Gamma(|\alpha|)}{2 C^{\alpha+1}} {\cal M}(|\alpha|)
\end{equation}
where the $w_x^{(\alpha)}$'s are defined in Formulae (\ref{ww}) and (\ref{www}).
\item When $\alpha=0$, for every $x,y\in\bbN$, under $\bbP_x$,
$$ \frac{N_n(y)}{\log n}\stackrel{\mathcal{L}}{\longrightarrow} \frac{(2y+1)}{4 C}\  {\cal E}(1)$$
where ${\cal E}(1)$ denotes the exponential distribution with parameter one. 
\end{itemize}
\end{theo}
\noindent{\bf Remark:}\\*
From (\ref{loc}), the limit distribution in Formula $(\ref{res1})$ is equal to the law of the random variable
$$\frac{(2y+2\alpha+1)\Gamma(y+2\alpha+1)}{\Gamma(y+1)(2C)^{\alpha+1}} L_1^{(\alpha)}(0)$$
where $L_1^{(\alpha)}(0)$ denotes the local time at 0 of the Bessel process with index $\alpha$.
\begin{proof}
Assume that $\alpha\in[-\frac{1}{2},0[$. For every $x,y\in\bbN$, we denote by $F_{x,y}$ the generating function of the sequence $\left(p^{(n)}(x,y)\right)_{n\geq 0}$,
namely for every $\lambda\in [0,1[$,
$$F_{x,y}(\lambda)= \sum_{n=0}^{\infty} \lambda^n p^{(n)}(x,y).$$
From Theorem \ref{llt}, for every $\varepsilon>0$, there exists $n_0$ such that for every $n\geq n_0$,
$$\frac{(1-\varepsilon)w_y^{(\alpha)} \Gamma(\alpha+1)}{2(Cn)^{\alpha+1}}  \leq p^{(n)}(x,y) \leq \frac{(1+\varepsilon)w_y^{(\alpha)} \Gamma(\alpha+1)}{2(Cn)^{\alpha+1}}.$$
From Tauberian theorem for power series (see Feller \cite{FE}, p. 447), we deduce that, as $\lambda\rightarrow 1^{-}$,
\begin{equation}\label{simi}
F_{x,y}(\lambda)\sim\frac{w_y^{(\alpha)} \Gamma(\alpha+1) \Gamma(|\alpha|)}{2 C^{(\alpha+1)}(1-\lambda)^{|\alpha|}}.
\end{equation}
Let $p\geq 1$, by combining all permutations of the sames indices $j_1,\ldots,j_p$, we have
\begin{eqnarray}\label{decom} 
\bbE_x(N_n(y)^p)&=& \sum_{0\leq j_1, \ldots, j_p\leq n} \bbP_x(S_{j_1}=\ldots=S_{j_p}=y)\nonumber \\
&=& p! \sum_{0\leq j_1\leq \ldots\leq j_p\leq n} p^{(j_1)}(x,y)p^{(j_2-j_1)}(y,y)\ldots p^{(j_p -j_{p-1})}(y,y)+R_n \ \ \ \ \ \ \ \, \, 
\end{eqnarray} 
The remainder term $R_n$ contains the sums over the $q$-tuples $(j_1,\ldots,j_q) \in\ \{0,\ldots,n\}^q$ with $q<p$. From Theorem $\ref{llt}$, we deduce that $R_n={\cal O}(n^{|\alpha| q})=o(n^{|\alpha|p})$, so it will be negligible in the limit. \\* 
We denote by $m_n:=m_n(x,p)$ the first sum in the right-hand side of (\ref{decom}) and by $G$ the generating function of the sequence $(m_n)_{n\geq 0}$, that is, for every $\lambda\in [0,1[$,
$$G(\lambda)= \sum_{n=0}^{\infty} \lambda^n m_n$$
which can be rewritten as 
\begin{eqnarray*}
G(\lambda)&=& p! \sum_{n=0}^{\infty} \sum_{0\leq j_1\leq \ldots\leq j_p\leq n}\left(\lambda^{j_1} p^{(j_1)}(x,y)\right) \left(\lambda^{j_2-j_1}p^{(j_2-j_1)}(y,y)\right)\ldots \left(\lambda^{j_p-j_{p-1}}p^{(j_p -j_{p-1})}(y,y)\right) \lambda^{n-j_p}\\
&=& p!  \sum_{n=0}^{\infty} \sum_{m_1+\ldots +m_{p+1}=n;m_i\geq 0}\left(\lambda^{m_1} p^{(m_1)}(x,y)\right) \left(\lambda^{m_2}p^{(m_2)}(y,y)\right)\ldots \left(\lambda^{m_p}p^{(m_p)}(y,y)\right) \lambda^{m_{p+1}}\\
&=&p! \left(\frac{1}{1-\lambda}\right) F_{x,y}(\lambda) \Big(F_{y,y}(\lambda)\Big)^{p-1}
\end{eqnarray*}
From (\ref{simi}), we deduce that, as $\lambda\rightarrow 1^{-}$, 
\begin{equation}
G(\lambda)\displaystyle\sim p! \frac{(w_y^{(\alpha)} \Gamma(\alpha+1) \Gamma(|\alpha|))^{p}}{2^p C^{p(\alpha+1)}(1-\lambda)^{|\alpha|p+1}}.
\end{equation} 
Then, from Tauberian theorem for power series (see Feller \cite{FE}, p. 447), we get as $n\rightarrow +\infty$,
$$ \bbE_x\left(\left(\frac{N_n(y)}{n^{|\alpha|}}\right)^p\right)\sim \frac{p!}{\Gamma(|\alpha|p+1)} \left(\frac{w_y^{(\alpha)} \Gamma(\alpha+1)\Gamma(|\alpha|)}{2C^{\alpha+1}}\right)^p=:\beta_p.$$
The Carleman condition 
$$\sum_{p=1}^{+\infty} \frac{1}{\beta_{2p}^{1/2p}}=+\infty$$
being satisfied, the limit distribution is uniquely determined and the weak convergence is proved. We characterize the limit distribution by recognizing the moments of the Mittag-Leffler distribution ${\cal M}(|\alpha|)$ (see Section 2).\\*
The proof in the case $\alpha=0$ is similar and is omitted.
\end{proof}
When $\mu=\delta_1$, the random walks associated with Gegenbauer polynomials are the birth and death Markov chains on $\bbN$ with transition probabilities $(p(i,j))_{i,j\in\bbsN}$ given by $p(0,1)=1 $ and
$$p(i,i+1)=\frac{i+2\alpha+1}{2i+2\alpha+1}; \ \ \  p(i,i-1)=\frac{i}{2i+2\alpha+1}.$$ 
When $\alpha\in [-1/2,0],$ the Markov chain is positive recurrent; we still denote by $(N_n(x))_{n\in\bbsN; x\in \bbsN}$ its local time. Thanks to Proposition $\ref{apero}$, we can adapt the proof of the previous theorem to provide a complete description of the limit behaviour of these local times. 
When $\alpha=-1/2$, the Markov chain corresponds to the simple random walk on $\bbN$ with reflection at 0. The mean number of times the Markov chain visits 0 is asymptotically equal to $\sqrt{n}$. The random walk associated with Gegenbauer polynomials with index $\alpha=0$ is the birth and death Markov chain on $\bbN$ with transition probabilities given by 
$$p(0,1)=1 $$
and
$$p(i,i+1)=\frac{i+1}{2i+1}, \ \ \ p(i,i-1)=\frac{i}{2i+1}.$$
In that case, the mean number of times the Markov chain visits 0 is asymptotically equal to $\frac{1}{2} \log(n)$.
More precisely, we have 
\begin{pr}\label{prlt}
\begin{itemize} 
\item When $\alpha\in[-\frac{1}{2},0[$, for every $x,y\in\bbN$, under the measure $\bbP_x$,
\begin{equation}
\frac{N_n(y)}{n^{|\alpha|}} \stackrel{\mathcal{L}}{\longrightarrow} 
\frac{(2y+2\alpha+1) \Gamma(y+2\alpha+1)\Gamma(|\alpha|)}{2^{\alpha+1} \Gamma(y+1) \Gamma(\alpha+1)} {\cal M}(|\alpha|).
\end{equation}
(with the convention $0\times \Gamma(0)=1$).
\item When $\alpha=0$, for every $x,y\in\bbN$, under the measure $\bbP_x$, 
\begin{equation}
\frac{N_n(y)}{\log(n)} \stackrel{\mathcal{L}}{\longrightarrow} {\cal E}\left(\frac{2}{2y+1}\right).
\end{equation}
\end{itemize}
\end{pr} 

\noindent{\bf Acknowledgement:} The author is very grateful to referees for their helpful comments.

\end{document}